\documentclass[reqno]{amsart}
\usepackage{fullpage}

\newtheorem{theorem}{Theorem}[section]
\newtheorem{lemma}[theorem]{Lemma}
\newtheorem{proposition}[theorem]{Proposition}

\newtheorem{thm}{Theorem}

\theoremstyle{definition}

\newtheorem{hypothesis}[theorem]{Hypothesis}

\theoremstyle{remark}
\newtheorem{remark}[theorem]{Remark}

\numberwithin{equation}{section}

\usepackage{amsmath,amsfonts,amssymb,amscd}
\usepackage{mathrsfs}
\usepackage{xcolor}

\DeclareMathOperator\Aut{Aut}

\DeclareMathOperator\codim{codim}

\DeclareMathOperator\depth{depth}

\DeclareMathOperator\edim{edim}
\DeclareMathOperator\End{End}
\DeclareMathOperator\Frac{Frac}
\DeclareMathOperator\gr{gr}
\DeclareMathOperator\GKdim{GKdim}
\DeclareMathOperator\gldim{gldim}
\DeclareMathOperator\height{height}
\DeclareMathOperator\Hom{Hom}

\DeclareMathOperator\Kdim{Kdim}
\DeclareMathOperator\lcm{lcm}
\DeclareMathOperator\lm{lm}
\DeclareMathOperator\LT{LT}
\DeclareMathOperator\Maxspec{MaxSpec}

\DeclareMathOperator\ord{ord}

\DeclareMathOperator\rank{rank}
\DeclareMathOperator\Spec{Spec}
\DeclareMathOperator\reg{reg}
\DeclareMathOperator\tr{tr}

\renewcommand\int{\mathrm{int}}

\newcommand\cnt{\mathcal Z}
\newcommand\detq{D}
\newcommand\inv{^{-1}}
\newcommand\ltR{\widetilde{R}}
\newcommand\nag{\Omega}

\newcommand\iso{\cong}
\newcommand\tensor{\otimes}
\newcommand\tornado{\xi}
\newcommand\kk{\Bbbk}

\newcommand\cO{\mathcal O}
\newcommand\cP{\mathcal P}

\newcommand\NN{\mathbb N}

\newcommand\ZZ{\mathbb Z}

\newcommand\ba{\mathbf a}
\newcommand\bb{\mathbf b}

\newcommand\bp{\mathbf p}
\newcommand\bq{\mathbf q}

\newcommand\fm{\mathfrak m}

\newcommand{\pma}{\cO_{\lambda,\bp}(M_n(\kk))}
\newcommand{\pmatwo}{\cO_{\lambda,\bp}(M_2(\kk))}
\newcommand{\pmathree}{\cO_{\lambda,\bp}(M_3(\kk))}
\newcommand{\sqman}{\cO_q(M_n(\kk))}
\newcommand{\sqmatwo}{\cO_q(M_2(\kk))}
\newcommand{\sqmathree}{\cO_q(M_3(\kk))}

\newcommand\restrict[1]{\raisebox{-.5ex}{$|$}_{#1}}

\begin{document}

\title{Centers and automorphisms of PI Quantum Matrix Algebras}

\author{Jason Gaddis}
\address{Miami University, Department of Mathematics, Oxford, Ohio 45056}
\email{gaddisj@miamioh.edu}

\author{Thomas Lamkin}
\address{Miami University, Department of Mathematics, Oxford, Ohio 45056}
\curraddr{Department of Mathematics, University of California San Diego, La Jolla, California 92903}
\email{lamkintd@miamioh.edu}
\thanks{Lamkin was partially supported by the Miami University Dean’s Scholar program.}

\date{July 24, 2022}



\begin{abstract}
We study PI quantum matrix algebras and their automorphisms using the noncommutative discriminant. In the multi-parameter case at $n=2$, we show that all automorphisms are graded when the center is a polynomial ring. In the single-parameter case, we determine a presentation of the center and show that the automorphism group is not graded, though we are able to describe certain families of automorphisms in this case, as well as those of certain subalgebras.
\end{abstract}

\maketitle

\section{Introduction}

In \cite{Y2}, Yakimov resolved the Launois-Lenagan conjecture \cite{LL1}, computing the automorphism group of a generic single-parameter quantum matrix algebra. In this work we are interested in a related question, but in the non-generic case. That is, we study automorphisms of single- and multi-parameter quantum matrix algebras at roots of unity.

Throughout, $\kk$ is an algebraically closed field of characteristic zero. All algebras are $\kk$-algebras and all tensor products should be regarded as over $\kk$. For $\tornado \in \kk^\times$, we write $\ord(\tornado)$ to denote the order of $\tornado$ as an element of the multiplicative group $\kk^\times$. The center of an algebra $A$ will be denoted $\cnt(A)$. 

Let $\lambda \in \kk^\times$ and let $\bp$ be a multipliciatively antisymmetric matrix over $\kk$ (i.e., $p_{ij}=p_{ji}\inv$ for all $i,j$). The \textit{($n\times n$) multi-parameter quantum matrix algebra}, $\pma$, 
is generated by $\{x_{ij}\}_{1 \leq i,j \leq n}$ subject to the relations
\[ x_{lm}x_{ij} = 	
\begin{cases}
	p_{li}p_{jm}x_{ij}x_{lm} + (\lambda-1)p_{li}x_{im}x_{lj} 
		& l > i, m > j\\
	\lambda p_{li}p_{jm} x_{ij}x_{lm} 
		& l > i, m \leq j \\
	p_{jm}x_{ij}x_{lm} 
		& l = i, m > j.			
\end{cases}\]

Let $q \in \kk^\times$. One can recover the single-parameter quantum matrix algebra $\sqman$ from the above definition by setting $p_{ij}=q$ for $i > j$ and $\lambda=q^{-2}$. The single-paramter and multi-parameter quantum matrix algebras are also related through cocycle twisting, which we review in the next section. 

For $q \in \kk^\times$ a nonroot of unity, Alev and Chamarie showed that $\Aut(\sqmatwo)$ is a semidirect product $(\kk^\times)^3 \rtimes \{\tau\}$ where $(\kk^\times)^3$ is isomorphic to the group of scalar automorphisms and $\tau$ is the transposition given by $\tau(x_{ij})=x_{ji}$ for all $1 \leq i,j \leq 2$ \cite[Theorem 2.3]{AC}. Launois and Lenagan proved that $\Aut(\sqmathree)$ has a similar description, thus motivating their conjecture that such a result should hold for $\Aut(\sqman)$. As mentioned above, this conjecture was resolved by Yakimov \cite{Y2}. See also \cite{LL2} for results in the case of \emph{rectangular} quantum matrix algebras.

Automorphisms of some quantum algebras at roots of unity have been considered previously, for example by Alev and Dumas \cite{AD1}. Our primary tool for studying automorphisms in the non-generic case will be noncommutative discriminants, as developed by Ceken, Palmieri, Wang, and Zhang \cite{CPWZ1,CPWZ2}. However, the rank of $\pma$ over its center is large and so it is difficult to compute this discriminant directly. 

There are several ways one may overcome this obstruction. One could view $\pma$ as a specialization and consider the Poisson structure on the center so as to employ the Poisson techniques of Nguyen, Trampel, and Yakimov \cite{NTY}. On the other hand, one could view $\pma$ as an iterated Ore extension and in this way invoke results of Kirkman, Moore, and the first author \cite{GKM}. More recently, Chan, Won, Zhang and the first author introduced and utilized the \emph{reflexive hull discriminant} in cases where an algebra is not free over its center \cite{CGWZ1}. Finally, one could use a property $\cP$-discriminant as introduced in \cite{LWZ_Discrim}. Our strategy will be a mixture of these.

In Section \ref{sec.back}, we review necessary prerequisites on quantum matrix algebras and noncommutative discriminants. Section \ref{sec.mpn2} studies automorphisms in the multi-parameter case.  

\begin{thm}[Theorems \ref{thm.n2auto} and \ref{thm.n3auto}]
Let $M=\pmatwo$ be PI. If $\cnt(M)$ is polynomial, then $\Aut(M)=\Aut_{\gr}(M)$.
\end{thm}

In addition we give explicit conditions for the center to be polynomial and compute the automorphism group, showing that the result aligns with that of Launois and Lenagan.

For the $n=3$ case, we do not know if there exist examples where the center is polynomial. However, under this hypothesis it is possible to show that all automorphisms are graded. We refer the interested reader to the appendix 
for that argument. Note that, in general, centers of PI AS regular algebras are quite complicated even for basic examples such as PI skew polynomial rings. See \cite{CPWZ2,CGWZ2} for partial progress on this question.

In the remainder of the paper, we study the single-parameter case $\sqman$ assuming $q$ is a root of unity. For much of this we will assume $\ord(q)=m \geq 3$ is an odd positive integer, but offer some general results as well. 

Section \ref{sec.spcnt} is a study on the center of $\sqman$. In our setting the generators of the center were computed by Jakobsen and Zhang \cite{JZ1}. We give a presentation of the center using Gr\"{o}bner basis techniques (Theorem \ref{thm.center_pres}). We are further able to derive several properties of the center.

\begin{thm}[Theorem \ref{thm.Zprops}]
Assume that $q$ is a root of unity such that $\ord(q)=m \geq 3$ is odd. Then $Z=\cnt(\sqman)$ is not Gorenstein and $\sqman$ is not projective over $Z$.
\end{thm}

In Section \ref{sec.spauto} we study automorphisms of $\sqman$. We show that the Launois-Lenagan conjecture does not extend to this setting.

\begin{thm}[Theorem \ref{thm.free_prod}]
Assume that $q$ is a root of unity. Then
$\sqman$ has non-graded automorphisms. Moreover, $\Aut(\sqman)$ contains a free product on two generators.
\end{thm}

We show additionally that $\Aut(\sqmatwo)$ contains an automorphism which restricts to a wild automorphism on a polynomial subalgebra (Theorem \ref{thm.wild_auto}). Using discriminants, we establish a certain ideal that is fixed by every automorphism (Proposition \ref{prop.P_discrim}). We conclude by studying certain subalgebras 
of $\sqmathree$.

\section{Background on quantum matrix algebras and discriminants}\label{sec.back}

In this section we provide some basic background and references for our main results and computations below. For an algebra $A$, we denote its Gelfand-Kirillov (GK) by $\GKdim(A)$ and its global dimension by $\gldim(A)$. If $A$ is commutative, we denote its Krull dimension by $\Kdim(A)$.

\subsection{Quantum matrix algebras}

Notice that if $\lambda=1$, then $\pma$ is isomorphic to some quantum affine space. On the other hand, when $\lambda=-1$, then $\pma$ does not have the same Hilbert series as a polynomial ring in $n^2$ variables. Hence, we assume throughout $\lambda^2 \neq 1$.

Recall that if $\Gamma=(\gamma_{ij})$ is a multiplicatively antisymmetric matrix, then the \emph{quantum affine space} $\mathcal{O}_{\Gamma}(\kk^n)$ is the algebra generated by $x_1,\hdots,x_n$ with relations $x_ix_j=\gamma_{ij} x_jx_i$ for all $i,j$. Quantum matrix algebras are related to quantum affine spaces in a way explained in the next proposition.

We begin by reviewing several well-known properties of $\pma$ and its center. Recall that a prime affine algebra $A$ which is finitely generated over its center is \emph{homologically homogeneous} if all simple $A$-modules have the same projective dimension over the center.

\begin{proposition}
Assume $\lambda^2\neq 1$. Let $M=\pma$. Then the following hold.
\begin{enumerate}
    \item \label{prop1} $M$ is Artin-Schelter regular with $\gldim(M)=\GKdim(M)=n^2$.
    \item \label{prop4} $M$ is PI if and only if $\lambda$ and the $p_{ij}$, $1 \leq i,j \leq n$, are roots of unity. 
    \item \label{prop5} If $M$ is PI, then $\cnt(M)$ is a Cohen-Macaulay (CM) normal Krull domain.
\end{enumerate}
\end{proposition}
\begin{proof}
\eqref{prop1} This follows from \cite[Theorem 2]{ASTgln}.

\eqref{prop4} This is clear by noting that $\Frac M$ is isomorphic to the quotient division ring of $\mathcal{O}_{\Gamma}(\kk^{n^2})$ where $\Gamma$ is the multiplicatively antisymmetric matrix whose entries are products of the $p_{ij}$ and $\lambda$ (\cite[Corollary 4.7]{haynal}).

\eqref{prop5}  Let $m=\lcm(\ord(p_{ij}),\ord(\lambda p_{ji}))$. Then it is straightforward to see that $C=\kk[x_{ij}^m \mid 1\leq i,j\leq n]$ is a central subalgebra of $M$. Moreover, $M$ is a free $C$-module with basis $\{\prod_{i,j=1}^nx_{ij}^{r_{ij}}\mid 0\leq r_{ij}<m\}$. Since $C$ is CM, $M$ is a CM $C$-module. Therefore, $M$ is homologically homogeneous by \cite[Lemma 2.4]{SVdb_Resolutions}. That $\cnt(M)$ is a normal CM domain then follows by \cite[Theorem 2.3(d)]{SVdb_Resolutions}. Since $M$ is a homologically homogeneous Noetherian PI domain, $Z$ is a Krull domain by \cite[Theorem 1.4(iii)]{SZ_HomProps}.
\end{proof}

The quantum determinant of $\pma$ is defined as
\begin{align}\label{eq.qdet}
D_{\lambda,\bp}
	= \sum_{\pi \in S_n} \left( \prod_{\substack{1 \leq i < j \leq n \\ \pi(i) > \pi(j)}} (-p_{\pi(i),\pi(j)}) \right) x_{1,\pi(1)} x_{2,\pi(2)} \cdots x_{n,\pi(n)},
\end{align}
where $l(\pi)$ is the number of inversions in $\pi$. 
By \cite[Theorem 3(a)]{ASTgln},
\begin{align}\label{eq.normal}
D_{\lambda,\bp} x_{ij} = \lambda^{j-i} \left( \prod_{\ell=1}^n p_{j\ell} p_{\ell i} \right) x_{ij} D_{\lambda,\bp}.
\end{align}
That is, $D_{\lambda,\bp}$ is a normal element in $\pma$. 

In the single-parameter case, \eqref{eq.qdet} reduces to
\begin{align}\label{eq.qdets}
    D_q=\sum_{\pi\in S_n}(-q)^{l(\pi)}x_{1,\pi(1)}x_{2,\pi(2)}\cdots x_{n,\pi(n)},
\end{align}
By \eqref{eq.normal} or \cite[Theorem 4.6.1]{PW}, $D_q$ is a central element in the single-parameter case. By \cite[Corollary 4.4.4]{PW}, we have the following quantum Laplace expansion formula in the single-parameter case,
\begin{align}\label{eq.laplace}
D_q = \sum_{j=1}^n (-q)^{j-i} x_{ij} A(i,j)
= \sum_{j=1}^n (-q)^{j-i} x_{ji} A(j,i).
\end{align}
We caution the reader that our convention differs from Parshall and Wang (we replace $q\inv$ by $q$).

For subsets $I,J \subset \{1,\hdots,n\}$ with $|I|=|J| \geq 1$, the subalgebra of $\pma$ generated by $\{ x_{ij} \mid i \in I, j \in J\}$ is another quantum matrix algebra and the quantum determinant of this subalgebra, denoted $D(I,J)$, is a \emph{quantum minor} of $\pma$.

Let $I,J \subset \{1,\hdots,n\}$ with $|I|=|J| =k \geq 1$. Set $\widetilde{I} = \{1,\hdots,n\} \backslash I$ and $\widetilde{J} = \{1,\hdots,n\} \backslash J$. For simplicity, let $\widetilde{i} = \widetilde{ \{i\} }$ in the case of a singleton. Set $A(I,J)=D(\widetilde{I},\widetilde{J})$ (and $A(i,j)$ in the case of singletons). For $1\leq t\leq n$, let
\begin{align}\label{eq.tdet}
D(t):=D(\{1,\ldots,t\},\{n-t+1,\ldots,n\}).
\end{align}

Cocycle twisting for $\pma$ was first established by Artin, Schelter, and Tate, though our review follows from Goodearl and Lenagan \cite{GL2}.

The algebra $\pma$ has a bigrading by $\ZZ^n \times \ZZ^n$ in which $x_{ij}$ has degree $(\epsilon_i,\epsilon_j)$. Choose $q$ such that $\lambda=q^{-2}$. We write $\ba,\bb \in \ZZ^n$ for $\ba=(a_i)$ and $\bb=(b_i)$. Define a map $c:\ZZ^n \times \ZZ^n \to \kk^\times$ by
\[ c(\ba,\bb) = \prod_{i > j} (qp_{ji})^{a_i b_j}.\] 
Then 
\[
c(\ba+\ba',\bb) = c(\ba,\bb)c(\ba',\bb) \quad\text{and}\quad
c(\ba,\bb+\bb') = c(\ba,\bb)c(\ba,\bb'),
\]
so that $c$ is a multiplicative bicharacter on $\ZZ^n$ and thus a $2$-cocycle. Then
\[ c(\epsilon_i,\epsilon_j) = \begin{cases} qp_{ji} & \text{ if $i>j$} \\ 1 & \text{ if $i \leq j$}.\end{cases}\]
Let $A_q = \sqman$. We use $c$ to twist the multiplication on $A_q$ to obtain an algebra $A_q'$. There is an isomorphism of graded vector spaces $A_q \to A_q'$ given by $a \mapsto a'$. The multiplication on $A_q'$ is given by 
\[ a'b' = c(u_1,v_1)\inv c(u_2,v_2) (ab)'\]
for homogeneous $a,b \in A_q$ of bidegrees $(u_1,u_2)$ and $(v_1,v_2)$, respectively. Then it is easily checked that $A_q' \iso \pma$.

Using the method of cocycle twists above we can easily establish a multi-parameter version of \eqref{eq.laplace}.

\subsection{Discriminants}

Applications of discriminants to the study of automorphism groups of noncommutative algebras were studied by Ceken, Palmieri, Wang, and Zhang \cite{CPWZ1,CPWZ2}. Since then, it has been implemented by various authors and many tools have been developed to aid in their computation, thus resolving a host of problems in noncommutative algebra and noncommutative invariant theory. 

Let $A$ be a $\kk$-algebra and $C$ a central subalgebra which is a domain and such that $A$ is finitely generated over $C$. In general, $A$ may not be free over $C$, so we pass to a localization $F$ of $C$ such that $A_F := A \tensor_C F$ is free (and finitely generated) over $F$. Set $w=\rank_F(A_F) < \infty$. Then left-multiplication gives a natural embedding $\lm: A \to A_F \to \End_F(A_F)$. By the hypotheses, $\End_F(A_F) \iso M_w(F)$ and we let $\tr_{\int}$ denote the usual matrix trace in $M_w(F)$. The \emph{regular trace} is then the composition
\[ \tr_{\reg}: A \xrightarrow{\lm} M_w(F) \xrightarrow{\tr_{\int}} F.\]
Here, we only consider the case that the algebra $A$ is free over $C$ of rank $w<\infty$, and the image of $\tr_{\reg}$ is in $C$. In this case, we can take a $C$-algebra basis $\{c_1,\hdots,c_w\}$ of $A$. Then the discriminant is defined as
\[ d(A/C) = \det(\tr_{\reg}(c_ic_j)_{i,j=1}^w) \in C.\]
Note that the discriminant is unique up to a unit in $C$. For any automorphism $\phi$ of $A$ which preserves $C$, the discriminant is preserved by $\phi$ up to a unit in $C$.

Now let $A$ be a prime $\kk$-algebra and let $Z$ be the center of $A$ so that $A$ is free over $Z$ and the image of $\tr$ is in $Z$. Suppose that $X=\Spec Z$ is an affine normal $\kk$-variety. Let $U$ be an open subset of $X$ such that $X\backslash U$ has codimension 2 in $X$. If if there exists an element $d \in Z$ such that the principal ideal $(d)$ of $Z$ agrees with $d(A/C)$ on $U$, then \cite[Lemma 2.3(2)]{CGWZ1} implies that $d(A/Z)=_{Z^\times} d$. 
In \cite{CGWZ1}, this was called the \emph{reflexive hull discriminant} and defined more generally in the context of the \emph{modified discriminant}. However, freeness implies that the modified discriminant and the discriminant agree. Consequently, the reflexive hull discriminant and the discriminant agree.

One can also define discriminants relative to some (algebraic) property. Let $A$ be an algebra and let $Z=\cnt(A)$. Let $\cP$ be a property defined for $\kk$-algebras that is invariant under algebra isomorphisms. Following \cite{LWZ_Discrim}, we define the \emph{$\cP$-locus} of $A$ to be
\[ L_{\cP}(A):=\{\fm\in\Maxspec(Z)\mid Z_{\fm} \text{ has property } \cP\},\]
and the \emph{$\cP$-discriminant ideal} to be 
\[ I_{\cP}(A):= \bigcap_{\fm\in L_{\cP}(A)} \fm \subseteq Z.\]
By \cite[Lemma 3.5]{LWZ_Discrim}, if $\phi:A\rightarrow B$ is an algebra isomorphism, then 
\[ \phi(L_{\cP}(A))=L_{\cP}(B) \quad\text{and}\quad \phi(I_{\cP}(A))=I_{\cP}(B).\]
Our definition of the $\cP$-locus differs slightly from that of \cite{LWZ_Discrim} in that we study localizations of the center instead of quotients of the algebra. 
This choice reflects the fact that, in the case of $\sqmatwo$, it is both easier and sufficient to determine useful information about the automorphism group using the above definition.

\section{The $2 \times 2$ multi-parameter case}\label{sec.mpn2}

In this section we compute the automorphism group for the $2 \times 2$ quantum matrix algebra in the multi-parameter setting. The relations in $\pmatwo$, explicitly, are
\begin{align*}
x_{12} x_{11} &= p_{12} x_{11}x_{12} & 
x_{21} x_{11} &= (\lambda p_{21}) x_{11}x_{21} \\
x_{22} x_{12} &= (\lambda p_{21}) x_{12} x_{22} &
x_{22} x_{21} &= p_{12} x_{21} x_{22} \\	
x_{21} x_{12} &= (\lambda p_{21}) p_{21} x_{12}x_{21} &
x_{22} x_{11} &= x_{11}x_{22} + (\lambda-1)p_{21} x_{12}x_{21}.
\end{align*}

The subalgebra $R$ of $\pmatwo$ generated by $\{x_{11},x_{12},x_{21}\}$ is a quantum affine $3$-space. Then $\pmatwo=R[x_{22};\sigma,\delta]$ where $\sigma$ is an automorphism of $R$ and $\delta$ a $\sigma$-derivation of $R$, and these are determined by
\begin{align*}
\sigma(x_{11}) &= x_{11}	&	\sigma(x_{12}) &= \lambda p_{21} x_{12}	& \sigma(x_{21}) &= p_{12} x_{21} \\
\delta(x_{11}) &= (\lambda-1)p_{21}x_{12}x_{21}	& \delta(x_{12})&=0	& \delta(x_{21}) &= 0.
\end{align*}
Note that
\[
\sigma(\delta(x_{11})) 
	= \sigma((\lambda-1)p_{21}x_{12}x_{21})
	= \lambda (\lambda-1)p_{21} (x_{12}x_{21})
	= \lambda\inv \delta(x_{11})
	= \lambda\inv \delta(\sigma x_{11}).
\]
So $\delta(\sigma (x_{11}))=\lambda \sigma(\delta(x_{11}))$. This holds trivially with $x_{11}$ replaced by $x_{12}$ or $x_{21}$, and so $(\sigma,\delta)$ is a (left) $\lambda$-skew derivation of $R$.

We will assume the following hypothesis throughout this section.

\begin{hypothesis}\label{hyp.gcd2}
Let $M=\pmatwo$ as above where $p_{12}$ and $\lambda p_{12}$ are roots of unity such that $\ord(p_{12})$ and $\ord(\lambda p_{21})$ are relatively prime. We set $\ell=\lcm(\ord(p_{12}),\ord(\lambda p_{21}))$ and let $y_{ij}=x_{ij}^\ell$ for $i,j \in \{1,2\}$. Throughout, let $Z=\cnt(M)$ and let $D=D_{\lambda,\bp}$.
\end{hypothesis}

\begin{lemma}\label{lem.center}
Assume Hypothesis \ref{hyp.gcd2}. Then $Z=\kk[y_{11},y_{12},y_{21},y_{22}]$.
\end{lemma} 
\begin{proof}
It is easy to show that $x_{12}^\ell$ and $x_{21}^\ell$ are central. To check $x_{22}^\ell$ is central, we apply \cite[Lemma 6.2]{G}. Note that $\delta^2(x_{11})=0$, and so we have
\[ x_{22}^\ell x_{11} = x_{11}x_{22}^\ell + [\ell]_\lambda \sigma^{\ell-1} \delta(x_{11}) x_{22}^{\ell-1} = x_{11}x_{22}^\ell\]
where the last equality holds because $\ord(\lambda)$ divides $\ell$. A similar proof shows that $x_{11}^\ell$ is central.

On the other hand, suppose $z \in Z$ and write
$z = \sum f_{ij} x_{11}^i x_{22}^j$ where each $f_{ij}$ is an element in the subalgebra generated by $x_{12}$ and $x_{21}$. Let $f_{uv}x_{11}^u x_{22}^v$ be the element of highest degree in $z$ according to the ordering $x_{22} > x_{11}$. Hence,
\[ 0 = [x_{11},z] = [x_{11}, f_{uv}] x_{11}^ux_{22}^v + \text{(lower degree terms in $x_{22}$)}.\]
Thus $[x_{11}, f_{uv}]=0$ and similarly $[f_{uv},x_{22}]=0$. Write $f_{uv} = \sum \gamma_{ij} x_{12}^i x_{21}^j$ with $\gamma_{ij} \in \kk$. Then
\begin{align*}
0 &= [x_{11}, f_{uv}] 
	= x_{11} \sum \gamma_{ij} (1-p_{12}^i (\lambda p_{21})^j) x_{12}^i x_{21}^j \\
0 &= [f_{uv}, x_{22}]
	= \sum \gamma_{ij} ( (\lambda p_{21})^i p_{12}^j - 1) x_{12}^ix_{21}^j x_{22}.
\end{align*}
By our hypothesis, for every $i,j$ with $\gamma_{ij} \neq 0$, the first equation gives that $p_{12}^i=1$ and $(\lambda p_{21})^j=1$. Similarly, the second equation gives $(\lambda p_{21})^i=1$ and $p_{12}^j=1$. Thus, for every $i$ with $\gamma_{ij} \neq 0$, $\ord(p_{12})$ and $\ord(\lambda p_{21})$ divide $i$, so $\ell \mid i$. Similarly, $\ell \mid j$.
That is, $f_{uv}$ is a polynomial in $x_{12}^\ell$ and $x_{21}^\ell$. Now,
\begin{align*}
0 &= [x_{12}, f_{uv}x_{11}^u x_{22}^v] = (p_{12}^u - (\lambda p_{21})^v)f_{uv}x_{11}^u x_{12} x_{22}^v \\
0 &= [x_{21}, f_{uv}x_{11}^u x_{22}^v] = ( (\lambda p_{21})^u - p_{12}^v)f_{uv}x_{11}^u x_{21} x_{22}^v.
\end{align*}
A similar argument to the above shows that $\ell \mid u$ and $\ell \mid v$.
\end{proof}

Our strategy will then be the following. We will consider a localization of $M$ which is isomorphic to a localization of a quantum affine space. The discriminant of this quantum affine space over its center is known (see \cite{CPWZ2}). We compute this on ``enough" localizations, then patch them together using the method discussed in Section \ref{sec.back}. We remark that using localizations to compute discriminants has also been employed in \cite{CYZ1,LY}.

\begin{lemma}\label{lem.local2}
Assume Hypothesis \ref{hyp.gcd2}. Let $Y=\{ y_{11}^k : k \geq 0\}$. Then $Y$ is an Ore set in $M$ and $MY\inv = M[y_{11}^{-1}]$ is isomorphic to a localization of a quantum affine space.
\end{lemma}
\begin{proof}
It is clear that $Y$ is an Ore set in $M$. Recall that $R$ is the subalgebra of $M$ generated by $\{x_{11},x_{12},x_{21}\}$. Note that $Y$ may also be regarded as an Ore set in $R$. Let $R'=RY\inv=R[y_{11}\inv]$. The automorphism $\sigma$ and the $\sigma$-derivation $\delta$ introduced above have a unique extension to $R'$ which, by an abuse of notation, we also call $\sigma$ and $\delta$, respectively. Hence we have
\[M' = (R[x_{22};\sigma,\delta])Y\inv = (R')[x_{22};\sigma,\delta].\]

The $\sigma$-derivation $\delta$ is inner if there exists $\omega$ such that $\delta(r)=\omega r - \sigma(r)\omega$ for all $r \in R'$. Note that $x_{11}$ is a unit in $R'$ (since $x_{11} (x_{11}^{m-1} (x_{11}^m)\inv ) = 1$). 
Set $\omega = p_{21} x_{11}\inv x_{12}x_{21} \in R'$. Then we have
\begin{align*}
\omega x_{11} - \sigma(x_{11})\omega 
	&= (p_{21} x_{11}\inv x_{12}x_{21})x_{11} - x_{11}(p_{21} x_{11}\inv x_{12}x_{21}) \\
	&= p_{21} (\lambda-1) x_{12}x_{21} = \delta(x_{11}) \\
\omega x_{12} - \sigma(x_{12})\omega
	&= (p_{21} x_{11}\inv x_{12}x_{21})x_{12} - \lambda p_{21} x_{12}(p_{21} x_{11}\inv x_{12}x_{21}) \\
	&= p_{21}  (\lambda p_{21}^2- (\lambda p_{21})p_{21}) (x_{11}\inv x_{12}^2x_{21}) = 0 \\
\omega x_{21} - \sigma(x_{21})\omega
	&=  (p_{21} x_{11}\inv x_{12}x_{21})x_{21} - p_{12} x_{21}(p_{21} x_{11}\inv x_{12}x_{21}) \\
	&= p_{21} (1-p_{12}(\lambda\inv p_{12})(\lambda p_{21}^2)) (x_{11}\inv x_{12}x_{21}^2)x_{21} = 0.
\end{align*}
It follows that for all $r \in R$,
\[ x_{22} r 
	= \sigma(r) x_{22} + \delta(r) 
	= \sigma(r) x_{22} + \omega r - \sigma(r)\omega 
\Rightarrow (x_{22}-\omega)r = \sigma(r) (x_{22}-\omega).\]
Thus, $M' = R'[x_{22};\sigma,\delta] =R'[x_{22}-\omega;\sigma]$.
\end{proof}

Next we compute the discriminant of $M$ over $Z$, which agrees with the formula for the single-parameter quantum matrix algebra case considered in \cite{NTY}. 

\begin{lemma}\label{lem.n2disc} 
Assume Hypothesis \ref{hyp.gcd2}. Let 
$\Omega := D^\ell = y_{11}y_{22}-y_{12}y_{21}$. Then
\[ d(M/Z)=_{\kk^\times} \left( x_{12}x_{21} D \right)^{\ell^4(\ell-1)}=_{\kk^\times} \left( y_{12}y_{21} \Omega \right)^{\ell^3(\ell-1)}.\]
\end{lemma}
\begin{proof}
Let $S=\kk_\bq[x_1,x_2,x_3,x_4]$ be the quantum affine space whose localization is isomorphic to $M'=M[y_{11}\inv]$ by Lemma \ref{lem.local2}. Here, $x_1,x_2,x_3,x_4$ represent the images of $x_{11}, x_{12}, x_{21}, x_{22}-\omega$, respectively, in $S$. Note that $S$ is free over the central subalgebra $C=\kk[x_1^\ell,x_2^\ell,x_3^\ell,x_4^\ell]$ and so by \cite[Proposition 2.8]{CPWZ2}, 
\[ d(S/C)=_{\kk^\times} (x_1 x_2 x_3 x_4)^{\ell^4(\ell-1)}.\]
Set $S'=S[y_1\inv]$ and let $C'=C[y_1\inv]$. Then $d(S'/C') =_{(C')^\times} d(S/C)$ by \cite[Lemma 1.3]{CYZ1}. Recall that  $M' \iso S'$. Set $Z'=Z[y_{11}\inv]$. Then
\[ d(M'/Z') =_{(Z')^\times} \left( x_{12}x_{21} D \right)^{\ell^4(\ell-1)}.\]

By a completely symmetric argument, localizing at powers of $y_{22}$ in place of $y_{11}$ gives the same result. It is left to show that this agrees with $d(M/Z)$. 

Clearly $M$ is a prime $\kk$-algebra and $X=\Spec Z$ is an affine $\kk$-variety. Thus, $(M,Z)$ satisfies \cite[Hypothesis 2.1]{CGWZ1}. Let $U_1$ and $U_2$ be the open subsets of $X$ with $y_{11} \neq 0$ and $y_{22} \neq 0$, respectively. Let $I$ be the ideal in $Z$ generated by $y_{11}$ and $y_{22}$. Then we have
\begin{align*}
\codim(X \backslash (U_1 \cup U_2) )
	&= \codim( (X\backslash U_1) \cap (X \backslash U_2) ) 
	= \codim( V(y_{11}) \cap V(y_{22}) ) \\
	&= \codim( V(y_{11},y_{22}) )
	= \codim(I) = \height(I) = 2.
\end{align*}
Thus, by \cite{CGWZ1}, the discriminant on $U_1 \cup U_2$ extends to a discriminant over $X$.
\end{proof}

We conclude this section by computing the automorphism group of $M$.

\begin{theorem}\label{thm.n2auto}
Assume Hypothesis \ref{hyp.gcd2}. Then $\Aut(M) = (\kk^\times)^3 \rtimes \{\tau\}$.
\end{theorem}
\begin{proof}
The same argument as in \cite[Proposition 5.10]{NTY} shows that $d(M/Z)$ is locally dominating. Now by \cite[Theorem 2.7]{CPWZ1}, $\Aut(M)$ is affine. It follows easily that $\Aut(M)$ is graded (see \cite[Proposition 4.2]{LL2}). That is, $\Aut(M)=\Aut_{\gr}(M)$.

Let $\phi \in \Aut(M)$. By \cite[Theorem 3.4]{Giso2}, the (homogeneous) degree one normal elements of $M$ are of the form $\alpha_1x_{12}+\alpha_2x_{21}$ for $\alpha_i \in \kk$. By the same argument as in \cite[Lemma 5.5]{Giso2}, the ideals $(x_{12})$ and $(x_{21})$ are either fixed or swapped by any automorphism. As the automorphism $\tau$ swaps them, then up to conjugation by $\tau$ we may assume $\phi$ fixes these ideals. Hence $\phi(x_{12})=b x_{12}$ and $\phi(x_{21})=c x_{21}$ for some $b,c \in \kk^\times$. Write,
\begin{align*}
    \phi(x_{11}) &= a_1 x_{11} + a_2 x_{12} + a_3 x_{21} + a_4 x_{22} \\
    \phi(x_{22}) &= d_1 x_{11} + d_2 x_{12} + d_3 x_{21} + d_4 x_{22},
\end{align*}
for some $a_i,d_i \in \kk$. Then
\begin{align*}
0   &= \phi(x_{12}x_{11}-p_{12}x_{11}x_{12}) \\
    &= b (a_2(1-p_{12})x_{12}^2 + a_3p_{12}(1-\lambda p_{21})x_{12}x_{21} + a_4(1-p_{21}^2)x_{12}x_{22}).
\end{align*}
Hence, $a_2=a_3=0$. Now 
\[ 0 = \phi(x_{21}x_{11}-\lambda p_{21}x_{11}x_{21})
    = ca_4(1-\lambda)x_{12}x_{22}.\]
Since we assume that $\lambda \neq 1$, this implies that $a_4=0$. Similarly one shows that $d_1=d_2=d_3=0$. Set $a=a_1$ and $d=d_1$. Then we have
\begin{align*}
bc(\lambda-1)p_{21}x_{12}x_{21}
    &= \phi(x_{22}x_{11}-x_{11}x_{22})
    = (dx_{22})(ax_{11}) - (ax_{11})(dx_{22}) \\
    &= ad (\lambda-1)p_{21}x_{12}x_{21}.
\end{align*}

Thus, we let $H$ denote the automorphisms of $M$ satisfying $\phi(x_{11})=a x_{11}$, $\phi(x_{12})=b x_{12}$, $\phi(x_{21})=c x_{21}$, and $\phi(x_{22})=d x_{22}$ satisfying $ad=bc$. Then $\Aut(M)=H \rtimes \{\tau\}$.
\end{proof}

\section{The center of $\sqman$}\label{sec.spcnt}

From here on, we study single-parameter quantum matrix algebras $\sqman$ at roots of unity. Recall that $\sqman$ is generated by $\{x_{ij}\}_{1\leq i,j\leq n}$ subject to the relations
\[ x_{ij}x_{kl} = 	
\begin{cases}
	x_{kl}x_{ij}+(q-q^{-1})x_{il}x_{kj}
	    & k > i, l > j \\
	qx_{kl}x_{ij} 
		& k > i, l = j \\
	qx_{kl}x_{ij}
	    & k = i, l > j \\
	x_{ij}x_{lm} 
		& k > i, l < j.			
\end{cases}\]
In the generic case, the center of $\sqman$ is $\kk[D_q]$ (\cite[Theorem 1.6]{NYM}, \cite{RTF}). 
Our interest is in the non-generic case. For most of this section we work under the following hypothesis.

\begin{hypothesis}\label{hyp.sqman}
Assume that $q$ is a root of unity such that $\ord(q)=m \geq 3$ is odd. Let $Z=\cnt(\sqman)$.
\end{hypothesis}

Our goal will be to give a full presentation for $Z$ in the setting of Hypothesis \ref{hyp.sqman}. Recall that  $\tau$ is the automorphism of $\sqman$ defined by $x_{ij}\mapsto x_{ji}$ and that the notation $D(t)$ was introduced in \eqref{eq.tdet}. 
Assuming Hypothesis \ref{hyp.sqman}, 
a result of Jakobsen and Zhang \cite[Theorem 6.2]{JZ1} shows that $Z$ is generated by
\begin{align}\label{eq.center_gens}
    \{x_{ij}^m,D_q,\detq(t)^r&\tau(\detq(n-t)^{m-r})\mid
    i,j,t=1,\ldots,n \text{ and } r=0,\ldots,m\}.
\end{align}
Jakobsen and Zhang also computed a generating set in the case of $m$ even. However, as it is significantly more complicated, we do not consider that problem here.

We begin in earnest by setting up some notation, and proving a technical lemma relating quantum minors. For a positive integer $N$, we use the notation $[N]=\{1,\hdots,N\}$. We will rely on some results of Scott \cite{Scott}. In the notation of that paper, $\detq(t)=\Delta_{[t],[n]-[n-t]}$ for $1 \leq t \leq n-1$. Hence, $\tau(\detq(t))=\Delta_{[n]-[n-t],[t]}$. Given $I,J\subseteq [n]$, we write $I\prec J$ if $i<j$ for all $i\in I$ and $j\in J$.

\begin{lemma}\label{lem.minors_comm}
Let $1\leq t,t'\leq n-1$. Then
\begin{enumerate}
    \item $\detq(t)\tau(\detq(t'))=\tau(\detq(t'))\detq(t)$;
    \item $\detq(t)\detq(t')=\detq(t')\detq(t)$;
\end{enumerate}
\end{lemma}
\begin{proof}
(1) This is trivial if $t'\leq n-t$ as then each term of $\detq(t)$ and each term of $\tau(\detq(t'))$ commute, so we may assume $t'>n-t$. Let
\begin{align*}
I &:= \{t'+1,\ldots,n,n+1,\ldots,n+t'\} \text{ and} \\
J &:= \{1,\ldots,n-t,2n-t+1,\ldots,2n\}.
\end{align*}
In the notation of \cite{Scott}, $I=S([n]-[n-t'],[t'])$ and $J=S([t],[n]-[n-t])$. We claim that $I$ and $J$ are weakly separated (\cite[Definition 2]{Scott}).

Since $|I|=|J|=n$ and $t'>n-t$, we have 
\begin{align*}
I-J &= \{t'+1,\ldots,2n-t\}, \text{ and} \\
J-I &= \{1,\ldots,n-t,n-t'+1,\ldots,2n\}=[n-t]\sqcup\{n+t'+1,\ldots,2n\}.
\end{align*}
Denoting $J'=[n-t]$ and $J''=\{n+t'+1,\ldots,2n\}$, we see $J'\prec I-J\prec J''$ once again using $t'>n-t$. Hence by \cite[Theorem 1]{Scott}, $\detq(t)$ and $\tau(\detq(t'))$ skew-commute by a factor of $q^c$ for some $c\in\ZZ$. By \cite[Theorem 2]{Scott}, the integer $c$ can be computed as
\[ c=|J''|-|J'|+\Big|[n]-[n-t']\Big|-\Big|[t]\Big|=n-t'-n+t+t'-t=0.\]
The result follows.

(2) WLOG, assume $t'\leq t$. Since $D_q$ is a central element in each $\sqman$, it follows that $\detq(t)$ is a central element in the subalgebra generated by the $x_{ij}$ with $1\leq i\leq t,n-t+1\leq j\leq n$. In particular, $\detq(t)$ commutes with $\detq(t')$.
\end{proof}

Our next goal is to determine the ideal of relations in $Z$. In order to simplify notation, denote 
\[ 
Z_{ij}:=x_{ij}^m, \quad 
D:=D_q, \quad 
Y_{tr}:=\detq(t)^r\tau(\detq(n-t)^{m-r})
\] 
for $1 \leq i,j,t \leq n$ and $0 \leq r \leq m$. Hence, $Z_{ij},Y_{tr}$, and $D$ generate $Z$ by \eqref{eq.center_gens}. By \cite[Proposition 2.12]{JZ2}, we have the relation $D^m=\det(Z_{ij})$. More generally, $\detq(t)^m=\det(A_t)$ and $\tau(\detq(t))^m=\det(B_t)$, where $A_t=(Z_{ij})_{1\leq i\leq t,n-t+1\leq j\leq n}$ and $B_t=(Z_{ij})_{n-t+1\leq i\leq n,1\leq j\leq t}$, for each $1\leq t\leq n$. Consequently the $Y_{t0}$ and $Y_{tm}$ are superfluous, and since $\detq(n)=D$ and $\detq(0)=1$, we only need the $Y_{tr}$ with $1\leq t\leq n-1$ and $1\leq r\leq m-1$ to generate $Z$. 

We now define a monomial ordering on these elements. Given a word $w$ in the alphabet 
\[ \{Z_{ij},Y_{tr},D\mid 1\leq i,j\leq n,\\ 1\leq t\leq n-1, 1\leq r\leq m-1\},\] 
let $w_Y$ denote the subword consisting of the $Y_{tr}$ in $w$, and let $w_Z$ be the quotient $w/w_Y$. For example, if $w=Y_{11}Y_{21}^2DZ_{12}Z_{13}Z_{14}$, then $w_Y=Y_{11}Y_{21}^2$ and $w_Z=DZ_{12}Z_{13}Z_{14}$. Now, given two such words $w$ and $w'$, we will say $w>w'$ if
\begin{itemize}
    \item $w_Y>_{deglex}w'_Y$ with respect to the ordering $Y_{11}>\cdots>Y_{1,m-1}>Y_{21}>\cdots>Y_{n-1,m-1}$, or
    \item $w_Y=w'_Y$ and $w_Z>_{lex}w'_Z$ with respect to the ordering $D>Z_{11}>\cdots>Z_{1n}>Z_{21}>\cdots>Z_{nn}$.
\end{itemize}

\begin{lemma}\label{lem.relns_GB}
Assume Hypothesis \ref{hyp.sqman}. Let $\leq$ denote the monomial ordering defined above. With respect to $\leq$, the following families of elements form a Gr\"{o}bner basis of the ideal they generate:
\begin{enumerate}
    \item $D^m-\det(Z_{ij})$,
    \item $Y_{ti}Y_{tj}-Y_{t,i+j}\det(B_{n-t})$ if $i+j<m$,
    \item $Y_{ti}Y_{tj}-\det(A_t)\det(B_{n-t})$ if $i+j=m$,
    \item $Y_{ti}Y_{tj}-Y_{t,i+j-m}\det(A_t)$ if $i+j>m$,
\end{enumerate}
for $1 \leq t \leq n-1$.
\end{lemma}
\begin{proof}
The relations hold by Lemma \ref{lem.minors_comm}. We use the diamond lemma to show they form a Gr\"{o}bner basis. Indeed, with respect to the monomial ordering $\leq$, one computes the following leading terms:
\begin{itemize}
    \item $\LT(D^m-\det(Z_{ij})=D^m$,
    \item $\LT(Y_{ti}Y_{tj}-Y_{t,i+j}\det(B_{n-t}))=Y_{ti}Y_{tj}$, for $i+j<m$,
    \item $\LT(Y_{ti}Y_{tj}-\det(A_t)\det(B_{n-t}))=Y_{ti}Y_{tj}$, for $i+j=m$,
    \item $\LT(Y_{ti}Y_{tj}-Y_{t,i+j-m}\det(A_t))=Y_{ti}Y_{tj}$, for $i+j>m$.
\end{itemize}
There are no inclusion ambiguities and the only overlap ambiguities are of the form $Y_{ti}Y_{tj}Y_{tk}$ for some $1\leq i,j,k\leq m-1$. Suppose $i+j+k < m$ and $m < i+j,j+k$. Then both $Y_{t,i+j}Y_{t,k}\det(B_{n-t})$ and $Y_{t,i}Y_{t,j+k}\det(B_{n-t})$ can be reduced to $Y_{t,i+j+k-m}\det(A_t)\det(B_{n-t})$. A similar reduction occurs in the case $i+j+k\leq m$. The other cases are left to the reader.
\end{proof}

We are now in a position to determine a presentation for $Z$. We keep the notation of Lemma \ref{lem.relns_GB} and let $I$ denote the ideal generated by the Gr\"{o}bner basis computed therein.

\begin{theorem}\label{thm.center_pres}
Assume Hypothesis \ref{hyp.sqman}. Let
\[
T = \kk[Z_{ij},Y_{tr},D \mid 1\leq i,j\leq n, 1\leq t\leq n-1,1\leq r\leq m-1]
\]
and let $R=T/I$. Then the following hold:
\begin{enumerate}
    \item \label{cp1} $\Kdim(R)=n^2$,
    \item \label{cp2} $R$ is an integral domain, and
    \item \label{cp3} $Z \iso R$.
\end{enumerate}
\end{theorem}
\begin{proof}
Since $\sqman$ is a finitely generated module over $Z$, $\GKdim(Z)=\GKdim(\sqman)=n^2$ (\cite[Proposition 5.5]{KL}). Therefore, $Z$ is an integral domain of Krull dimension $n^2$ (\cite[Theorem 4.5]{KL}), and there is an obvious surjective homomorphism from $R$ onto $Z$. Thus, \eqref{cp3} follows from \eqref{cp1} and \eqref{cp2}.

Throughout we abuse notation and identify elements in $\kk[Z_{ij},Y_{tr},D]$ with their respective images in the quotient $R$.

\eqref{cp1} By \cite[Theorem 5.6.36 and Corollary 5.7.10(a)]{KR_CommAlg}, it suffices to
compute the Krull dimension of the graded ring $\ltR=T/\LT(I)$, which we do by computing the degree of the Hilbert polynomial of $\ltR$.

By Lemma \ref{lem.relns_GB}, we need to count the number of (monic) monomials in the variables $Z_{ij},Y_{tr},D$ of a given degree containing at most $m-1$ copies of $D$, and containing at most one $Y_{tr}$ for each $1\leq t\leq n-1$. A simple counting argument shows that the total number of such monomials containing no $Z_{ij}$ is
\[ m\sum_{i=0}^{n-1}\binom{n-1}{i}(m-1)^i=m^n,\]
with the equality being the binomial theorem applied to $((m-1)+1)^{n-1}$. Thus for sufficiently large degrees $N$, the number of ``good" monomials of degree $N$ is asymptotically equal to the number of monomials in the $Z_{ij}$ of degree $N$. This is well-known to be $\binom{N+n^2-1}{N}$, which is a degree $n^2-1$ polynomial in $N$. In other words, the Hilbert polynomial of $\ltR$ has degree $n^2-1$, and thus $\Kdim(\ltR)=n^2$ (\cite[Theorem 5.4.15(b)]{KR_CommAlg}).

\eqref{cp2} Before we prove that $R$ is a domain, first note that any nonzero polynomial in the $Z_{ij}$ is a regular element of $R$. Indeed, suppose $f(Z_{ij})g\in I$ for some $g\in T$. Then by Lemma \ref{lem.relns_GB},
\[ \LT(fg)=\LT(f)\LT(g)\in \LT(I)=(D^m,Y_{ti}Y_{tj} \mid 1\leq t\leq n-1, 1\leq i,j\leq m-1).\]
Since $LT(f)$ is a monomial in the $Z_{ij}$, it must be that $\LT(g)\in LT(I)$, so that $\LT(g)=\LT(g')$ for some $g'\in I$. Then $f(g-g')\in I$ with $\LT(g-g')<\LT(g)$. Since $<$ is a well-order, we can repeat this argument until we have written $g$ as a sum of elements in $I$, proving $g\in I$ and $f$ is not a zero divisor in $R$.

Let $S=\kk[Z_{ij} \mid i,j \in 1,\hdots,n]\backslash\{0\}$. By the preceding discussion, $S$ consists of regular elements and so it suffices to prove that the localization $Q=RS\inv$ is a domain.

First note that for $1\leq t\leq n-1$ and $1\leq k\leq m-2$, we have the relation $Y_{tk}=Y_{t,k+1}Y_{t,m-1}\det(A_t)^{-m+k+1}$ in $Q$. By repeatedly applying this relation, we obtain $Y_{tk}=Y_{t,m-1}^{m-k}\det(A_t)^{-m+k+1}$. Conversely, these relations imply $Y_{ti}Y_{tj}=Y_{t,i+j}\det(A_t)$ for $i+j<m$. Moreover, the relations $Y_{t1}Y_{t,m-1}=\det(A_t)\det(B_{n-t})$ can be rewritten as $Y_{t,m-1}^m=\det(A_t)^{m-1}\det(B_{n-t})$, from which the rest of the relations of the form $Y_{ti}Y_{tj}=\det(A_t)\det(B_{n-t})$ with $i+j=m$ follow. In addition, the above formulas for $Y_{t,m-1}^m$ and $Y_{tk}$ for $k<m-2$ imply $Y_{t,i}Y_{t,j}=Y_{t,i+j}\det(B_{n-t})$ for $i+j<m$. Next, note that since $\det(B_{n-t})=Z_{nt}f+g$ for some $f\in S$ and some $g\in S\cup\{0\}$, we can rewrite the relations $Y_{t,m-1}^m=\det(A_t)^{m-1}\det(B_{n-t})$ as $(Y_{t,m-1}^m\det(A_t)^{1-m}-g)f^{-1}=Z_{nt}$. Similarly, we can rewrite the relation $D^m=\det(Z_{ij})$ to solve for $Z_{nn}$.

It follows that $Q$ is (isomorphic to) a localization of the polynomial ring 
\[ \kk[X_{ij},Y_{t,m-1},D \mid 1\leq i,t\leq n-1, 1\leq j\leq n].\] 
Thus, $Q$ is a domain.
\end{proof}

Having determined a presentation for the center $Z$ of $\sqman$, we are now ready to discuss several properties of $Z$. 

\begin{theorem}\label{thm.Zprops}
Assume Hypothesis \ref{hyp.sqman}. Then $Z$ is not Gorenstein. Moreover, $\sqman$ is not projective over $Z$.
\end{theorem}
\begin{proof}
We show the localization of $Z$ at $\fm=(Z_{ij},Y_{tr},D)$ is not Gorenstein. Since $\{Z_{ij}\mid 1\leq i,j\leq n\}$ is a system of parameters for $Z_\fm$, it suffices to show $A=Z_\fm/(Z_{ij})$ is not Gorenstein by \cite[Theorem 2.1.2(d) and Proposition 3.1.19(b)]{BH_CM}. As $\depth(A)=0$, this is equivalent to showing $\dim_\kk\Hom_A(\kk,A)\geq 2$ by \cite[Lemma 1.2.19 and Theorem 3.2.10]{BH_CM}. But this is simple to see as
\[
    \phi:1\mapsto D^{m-1}\prod_{t=1}^{n-1}Y_{t1} 
    \quad\text{and}\quad
    \psi:1\mapsto D^{m-1}\prod_{t=1}^{n-1}Y_{t2}
\]
are clearly $\kk$-linearly independent $A$-algebra homomorphisms from $\kk$ to $A$.
 
Suppose $\sqman$ is projective over $Z$, so that there is a projective basis $\{z_i\mid i\in I\}$ with corresponding $Z$-module homomorphisms $f_i:\sqman\rightarrow Z$. Then, in particular, $x_{1n}=\sum_{i\in I}f_i(x_{1n})z_i$. Since $x_{1n}$ is not a term in any central element, a simple graded argument shows that some $f_j(x_{1n})$ is a nonzero scalar. However, since $x_{1n}\tau(\detq(n-1))^m=Y_{11}\tau(\detq(n-1))$, we have
\begin{align*}
\tau(\detq(n-1)^m)f_j(x_{1n})
    &=f_j(x_{1n}\tau(\detq(n-1))^m)
    =f_j(Y_{11}\tau(\detq(n-1))) \\
    &=Y_{11}f_j(\tau(\detq(n-1))).
\end{align*}
But this is impossible, since $x_{1n}$ divides the (possibly 0) RHS, but not the LHS.
\end{proof}

As a consequence of Theorem \ref{thm.Zprops}, $\sqman$ is not Azumaya over its center.

\section{Automorphisms of $\sqman$}\label{sec.spauto}

As mentioned earlier, in the case where $q$ is not a root of unity, every automorphism of $\sqman$ is graded. Similarly, as shown in Theorem \ref{thm.n2auto}, one can find hypotheses under which every automorphism of $\pmatwo$ is graded.
In this section, we will show that in the non-generic single-parameter case, the automorphism group of $\sqman$ contains a free group on two generators. Moreover, if $n=2$, we construct an automorphism of $\sqmatwo$ which induces a wild automorphism of a three variable polynomial ring. Finally, in the setting of Hypothesis \ref{hyp.sqman}, we use discriminants to compute an ideal of $\sqmatwo$ fixed by all automorphisms, and we investigate some natural subalgebras of $\sqmathree$, determining whether they have non-graded automorphisms as well. We do not assume Hypothesis \ref{hyp.sqman} in general for this section. Several results indeed hold for any root of unity $q$ and, in one case, for any parameter $q$.

The next result establishes both that $\sqman$ contains non-graded automorphisms and that its automorphism group contains a free product on two generators.

\begin{theorem}\label{thm.free_prod}
Let $q$ be any root of unity, and let $\ord(q)=m$. Then there are automorphisms $\phi$ and $\psi$ of $\sqman$ given by
\begin{align*}
    \phi:x_{ij}&\mapsto
    \begin{cases}
        x_{11}+A(1,1)^{m-1} & i=j=1\\
        x_{ij} & (i,j)\neq (1,1),
    \end{cases} \\
    \psi:x_{ij}&\mapsto
    \begin{cases}
    x_{ij} & (i,j)\neq (n,n)\\
    x_{nn}+A(n,n)^{m-1} & i=j=n.
    \end{cases}
\end{align*} 
Moreover, the proper subgroup of $\Aut(\sqman)$ generated by $\phi$ and $\psi$ is isomorphic to a free group on two generators.
\end{theorem}
\begin{proof}
We only show $\phi$ is an automorphism, as the proof for $\psi$ is similar.  The map
\begin{align*}
    \rho:x_{ij}\mapsto
    \begin{cases}
        x_{11}-A(1,1)^{m-1} & i=j=1\\
        x_{ij}, & (i,j)\neq (1,1),
    \end{cases}
\end{align*}
is an inverse to $\phi$, so $\phi$ is bijective. To show that $\phi$ is a homomorphism, notice that we only have to show $\phi(x_{11}x_{ij}-qx_{ij}x_{11})=0$, for either $i=1$ or $j=1$, and $\phi(x_{11}x_{ij}-x_{ij}x_{11}-(q-q^{-1})x_{i1}x_{1j})=0$, for $2\leq i,j\leq n$, as the other relations trivially map to 0. For the latter type, note that $x_{ij}$ commutes with $A(1,1)$ for each $2\leq i,j\leq n$ as, by definition, $A(1,1)$ is the quantum determinant of the quantum matrix subalgebra generated by these $x_{ij}$. On the other hand, for the former relation type, \cite[Lemma 4.5.1]{PW} shows that $A(1,1)^{m-1}x_{ij}=qx_{ij}A(1,1)^{m-1}$ for exactly one of $i,j$ equal to 1. Therefore, both relation types indeed map to zero.

We now prove that the subgroup generated by $\phi$ and $\psi$ is isomorphic to the free group $\langle\phi\rangle*\langle\psi\rangle\cong\ZZ*\ZZ$. We will show that every word $w=\psi^{m_r}\phi^{n_r}\cdots\psi^{m_1}\phi^{n_1}$ is not the identity automorphism, where $r\geq 1$ and where the $n_i,m_i\in\ZZ$ are not all 0. Let $\leq_1$ denote the lexicographical order with respect to the ordering $x_{11}<\cdots<x_{1n}<x_{21}\cdots<x_{nn}$, and let $\leq_2$ denote the reverse lexicographical order with respect to the same ordering of the $x_{ij}$.

First, notice that for $r,s\in\ZZ$,
\begin{align*}
    \phi^r:x_{ij}&\mapsto
    \begin{cases}
        x_{11}+rA(1,1)^{m-1} & i=j=1\\
        x_{ij} & (i,j)\neq (1,1),
    \end{cases} \\
    \psi^s:x_{ij}&\mapsto
    \begin{cases}
    x_{ij} & (i,j)\neq (n,n)\\
    x_{nn}+sA(n,n)^{m-1} & i=j=n.
    \end{cases}
\end{align*}
Now, the key observation is that if $\LT_{\leq_2}(f)=c\prod_{i,j=1}^nx_{ij}^{a_{ij}}$ for $f\in\sqman$, $c\in\kk$, and $a_{ij}\in\NN$, then 
\begin{align*}
\LT_{\leq_1}(\phi^r(f))=cr\left(\prod\limits_{\substack{i,j=1 \\ (i,j)\neq (1,1)}}^nx_{ij}^{a_{ij}}\right)\prod_{k=2}^nx_{kk}^{a_{11}(m-1)}.
\end{align*}
Likewise, if $\LT_{\leq_1}(g)=d\prod_{i,j=1}^nx_{ij}^{b_{ij}}$ for $g\in\sqman$, $d\in\kk$, and $b_{ij}\in\NN$, then
\begin{align*}
\LT_{\leq_2}(\psi^s(g))=ds\left(\prod\limits_{\substack{i,j=1 \\ (i,j)\neq (n,n)}}^nx_{ij}^{b_{ij}}\right)\prod_{k=1}^{n-1}x_{kk}^{b_{11}(m-1)}.
\end{align*}
It follows by induction that $w(x_{11})\neq x_{11}$ if $w$ is not a power of $\psi$, and that $w(x_{nn})\neq x_{nn}$ if $w$ is not a power of $\phi$. Since $\phi$ and $\psi$ are both infinite order automorphisms, $w$ is not the identity on $\sqman$, as desired.

Finally, note that the subgroup generated by $\phi$ and $\psi$ is proper since it does not contain, for instance, the transpose automorphism $\tau:x_{ij}\rightarrow x_{ji}$.
\end{proof}

\subsection{Automorphisms of $\sqmatwo$}

In this subsection alone, we restrict to the case $n=2$ and adopt the following notation changes: 
\begin{enumerate}
    \item $x_{11},x_{12},x_{21},x_{22}:=a,b,c,d$, respectively;
    \item $Z_{11},Z_{12},Z_{21},Z_{22}:=u,v,w,z$, respectively;
    \item $Y_{1r}=b^rc^{m-r}:=t_r$ for $1\leq r\leq m-1$.
\end{enumerate}
Though we continue to assume that $q$ is a root of unity, we do not immediately assume that $\ord(q)\geq 3$ is an odd integer.

Before proceeding, first recall that an automorphism of a polynomial ring $A=F[x_1,\ldots,x_t]$ over a field $F$ is \emph{elementary} if it is of the form
\begin{align*}
    (x_1,\ldots,x_{i-1},x_i,x_{i+1},\ldots,x_n)\mapsto (x_1,\ldots,x_{i-1},\alpha x_i+f,x_{i+1},\ldots,x_n),
\end{align*}
for some $0\neq\alpha\in F$ and $f\in F[x_1,\ldots,\hat{x_i},\ldots,x_n]$. An automorphism of $A$ is then \emph{tame} if it is a composition of elementary automorphisms. An automorphism which is not tame is called \emph{wild}.

If $n=1,2$, then it is well-known that every automorphism of $A$ is tame (\cite{Jung_Autos, Kulk_autos}). When $n=3$, Umirbaev and Shestakov proved the \emph{Nagata automorphism}, constructed by Nagata in \cite[\S 2.1]{Nagata_autos}, is wild in the characteristic zero case (\cite{SU_Wild}). The following theorem can be seen as an analogue of this result for $\sqmatwo$, as the automorphism therein restricts to an automorphism of $k[b,c,u]$ inspired by the Nagata automorphism, which is then shown to be wild using the methods of \cite{SU_Wild}.

\begin{theorem}\label{thm.wild_auto}
Let $q$ be a root of unity of with $\ord(q)=m$.
Let $\nag=cu+b^{m+1}$. Then the map
\begin{align*}
    \sigma: a &\mapsto a\\
    b &\mapsto b+\nag u^2\\
    c &\mapsto c-\sum_{i=1}^{m+1}\binom{m+1}{i}b^{m+1-i}\nag^iu^{2i-1}\\
    d &\mapsto d+qa^{m-1}\left(c\nag u^2-\sum_{i=1}^{m+1}\binom{m+1}{i}b^{m+1-i}\nag^iu^{2(i-1)}(b+\nag u^2)\right),
\end{align*}
extends to an automorphism of $\sqmatwo$. Moreover, the restriction $\sigma\restrict {\kk[b,c,u]}$ is a wild automorphism.
\end{theorem}
\begin{proof}
That $\sigma$ is injective is easy to see, while surjectivity follows from the observation that $\sigma(\nag)=\nag$. Moreover, that $\sigma$ sends all relations to zero is simple to check, proving $\sigma$ is an automorphism.

For an element $f\in \kk[b,c,u]$, denote the highest homogeneous part of $f$ with respect to the standard grading of $\kk[b,c,u]$ by $\overline{f}$. We show that none of $\overline{\sigma(b)},\overline{\sigma(c)},\overline{\sigma(u)}$ are contained in the subalgebra generated by the other two; that $\sigma\restrict{\kk[b,c,u]}$ is wild then follows by \cite[Corollary 8]{SU_Wild}. Indeed, we have that
\begin{align*}
\overline{\sigma(b)}=b^{m+1}u^2, \quad 
\overline{\sigma(c)}=b^{(m+1)^2}u^{2m+1}, \quad 
\overline{\sigma(u)}=u,
\end{align*}
from which it is clear none of these elements are contained in the subalgebra generated by the other two.
\end{proof}

The next result requires our presentation of the center from Section \ref{sec.spcnt}. Hence, we reinvoke Hypothesis 5.1. That is, we assume that $\ord(q)=m\geq 3$ is an odd integer. Using discriminants, we determine an ideal of $\sqmatwo$ that is fixed by every automorphism.

Recall the Jacobian criterion provides a simple way to determine whether the localization of an affine $\kk$-algebra $R=k[x_1,\ldots,x_r]/I$ at a maximal ideal is regular. In the course of proving the criterion however, even more is established: if $J_\fm$ denotes the Jacobian matrix of $I$ evaluated mod $\fm$, then
\[\rank(J_\fm)=r-\dim_\kk(M/M^2)=r-\edim(R_\fm),\]
where $M$ is the maximal ideal of $R_\fm$ and $\edim(R_\fm)$ is the embedding dimension of $R_\fm$ \cite[Theorem 5.6.12 and its proof]{GPBLS_Singular}. It follows that $\rank(J_\fm)$ is an easy to compute invariant of the algebras $R_\fm$. This motivates the choice of $\cP$ in the following proposition.

\begin{proposition}\label{prop.P_discrim}
Assume Hypothesis \ref{hyp.sqman}. Let $\cP$ be the property of being a local ring with embedding dimension $m+3$. Then the $\cP$-discriminant ideal of $\sqmatwo$ is $(v,w,t_1,\ldots,t_{m-1})$. Consequently, if $\sigma\in\Aut(\sqmatwo)$, then $\sigma((b,c))=(b,c)$.
\end{proposition}
\begin{proof}
By the above discussion and Theorem \ref{thm.center_pres}, the embedding dimension condition on a localization $Z_\fm$ of $Z$ at some $\fm\in\Maxspec(Z)$ is equivalent to $\rank(J_\fm)=1$. Again by Theorem \ref{thm.center_pres}, the ideal of relations $I$ of $Z$ is generated by the following families of elements:
\begin{enumerate}
    \item $D^m-uz+vw$,
    \item $t_it_j-t_{i+j}w$ if $i+j<m$,
    \item $t_it_j-vw$ if $i+j=m$, and
    \item $t_it_j-t_{i+j-m}v$ if $i+j>m$.
\end{enumerate}
Since $\kk$ is algebraically closed, to compute $J_\fm$ we need only compute the Jacobian matrix of the generators of $I$, then evaluate at the point corresponding to $\fm$ in the variety $X$ whose coordinate ring is $Z$.

Let $P$ be a point in $X$ and denote the coordinates of $P$ simply by their respective variables; e.g., the $v$-coordinate of $P$ is denoted just $v$. Suppose that $t_i=0$ for some $i$. Then we may assume $i=1$ as if $i>1$, we have $t_1t_{i-1}=t_iw=0$ by the second relation family, from which $t_1=0$ follows by induction. Now, if $w\neq 0$, then $t_i=0$ for all $i$ by induction and the second family of relations. On the other hand, if $w=0$, then $t_i^2=0$ for all $i<\frac{m}{2}$ by the second family. Hence, by the fourth family and induction, $t_i^2=0$ for all $i>\frac{m}{2}$ so that each $t_i=0$. Therefore, either all $t_i$ are 0, or no $t_i$ is 0. In these cases, we have $vw=0$ or $v,w\neq 0$, respectively, by the third relation family.

Now, suppose $P$ corresponds to a maximal ideal $\fm$ with $\rank(J_\fm)=1$. Then by the above discussion, each $t_i=0$, and so $vw=0$. Next, notice that if $w\neq 0$, then $J_\fm$ contains the rank 2 submatrix
\[
\begin{bmatrix}
    -z & w & 0 & -u & mD^{m-1} & 0 & 0 & 0 & \hdots & 0\\
    0 & 0 & 0 & 0 & 0 & 0 & -w & 0 & \hdots & 0
\end{bmatrix}.
\]
Similarly, if $v\neq 0$, then $\rank(J_\fm)\geq 2$. Finally, if $u=z=0$, then $D^m=uz-vw=0$, so that $\rank(J_\fm)=0$. We conclude that $\rank(J_\fm)=1$ if and only if $\fm=(u-a,v,w,z-b,D-c,t_1,t_2,\ldots,t_{m-1})$ for any $(a,b)\neq (0,0)$ with $c^m=ab$. Taking the intersection of these maximal ideals, we see that the $\cP$-discriminant ideal of $\sqmatwo$ is $(v,w,t_1,\ldots,t_{m-1})$.

Let $\sigma\in\Aut(\sqmatwo)$. We claim $\sigma$ fixes the ideal $(b,c)$. By \cite[Lemma 3.5]{LWZ_Discrim}, $\sigma$ fixes the ideal $(v,w,t_1,\ldots,t_{m-1})$. Notice that \[(b,c)^m=(v,w,t_!,\ldots,t_{m-1}).\] 
Moreover, since $\sqmatwo/(b,c)\cong k[a,d]$, the ideal $(b,c)$ is completely prime. Hence, if $x\in (b,c)$, then
\[ \sigma(x)^m=\sigma(x^m)\in (v,w,t_1,\ldots,t_{m-1})\subseteq(b,c),\]
from which it follows that $\sigma(x)\in (b,c)$.
\end{proof}

\subsection{Automorphisms of certain subalgebras of $\sqmathree$}

In this section we study automorphisms of certain subalgebras of $\sqmathree$. For a subset a $S\subseteq\sqmathree$, we let $\langle S \rangle$ denote the subalgebra of $\sqmathree$ generated by $S$. The particular subalgebras we study are listed below:
\[
B_1=\langle x_{ij}\rangle_{1\leq i\leq 2, 1\leq j\leq 3}, \quad
B_2=\langle B_1\cup x_{31}\rangle, \quad
B_3=\langle B_2\cup x_{32}\rangle, \quad
C=\langle x_{ij}\rangle_{i+j\geq 4}.
\]

The subalgebra $B_1$ is isomorphic to $\cO_q(M_{2,3}(\kk))$, i.e., the $2 \times 3$ quantum matrix algebra \cite{LL2}. Below we show that $B_1$ has polynomial center. Using results of Nguyen, Trampel, and Yakimov \cite{NTY}, we are then able to conclude that all automorphisms of $B_1$ are graded.

The algebra $B_3$ appears in \cite{JJ_QMA}, therein denoted $A_{n,r}$, in the authors' study of the algebras obtained by factoring out the ideal of $\sqman$ generated all $(r+1)\times (r+1)$ quantum subdeterminants. The algebra $B_2$ acts as an intermediary between $B_1$ and $B_3$. We will show that both $B_2$ and $B_3$ have non-graded automorphisms.

Like $B_1$, $C$ is a 6-generated subalgebra of $\sqmathree$ that has polynomial center under our standard hypotheses, as shown below, such that all automorphisms are graded. We do not know if there are any $7$- or $8$-generated subalgebras of $\sqmathree$ which have polynomial centers, though we suspect not.

The following proof is motivated in part by \cite{JZ1}.

\begin{proposition}\label{prop.B1_center}
Assume Hypothesis \ref{hyp.sqman}. The center of $B_1$ is polynomial. Consequently, $\Aut(B_1)=\Aut_{\gr}(B_1)$.
\end{proposition}
\begin{proof}
We will prove that $\cnt(B_1)=\kk[Z_{ij}]$. By \cite[Proposition 5.10]{NTY}, the discriminant of $B_1$ over $\kk[Z_{ij}]$ is locally dominating. Then \cite[Theorem 2.7]{CPWZ1} implies that all all automorphisms of $B_1$ are graded.

Let $\overline{B_1}$ denote the associated quasipolynomial algebra of $B_1$ \cite{DCP}; that is, if we write 
\[x_{ij}x_{kl}=q^{a_{(i,j),(k,l)}}x_{kl}x_{ij}+P_{(i,j),(k,l)}\] 
in $B_1$, then $\overline{B_1}$ is the skew-polynomial algebra generated by $w_{ij}$ with $1\leq i\leq 2,1\leq j\leq 3$, subject to the relations
\[ w_{ij}w_{kl}=q^{a_{(i,j),(k,l)}}w_{kl}w_{ij}.\]
Now, if we order the basis elements of $B_1$ via the lexicographical order with respect to $x_{11}>x_{12}>x_{13}>x_{21}>x_{22}>x_{23}$, then we can define a (nonlinear) map 
$T:B_1 \rightarrow \overline{B_1}$ by $P \mapsto \LT(P)$. By our choice of order, $T$ is multiplicative, from which it follows that $T$ maps central elements to central elements. Now, a simple computation reveals that $\overline{B_1}$ has polynomial center $\cnt(\overline{B_1})=\kk[w_{ij}^m]$. Therefore, for any element $Y\in \cnt(B_1)$, there is a polynomial $f\in\kk[Z_{ij}]$ such that $\LT(f)=\LT(Y)$. Hence $\LT(Y-f)<\LT(Y)$, so that $Y-f\in\kk[Z_{ij}]$ by induction. 
\end{proof}

Non-graded automorphisms are introduced in extending $B_1$ to $B_2$ and $B_3$. An example of such an automorphism is given in the next proposition. This result does not require $q$ to be a root of unity, but that hypothesis will be necessary in the latter part of the subsequent remark.

\begin{proposition}\label{prop.bautos}
Define $\phi:B_3\rightarrow B_3$ by
\begin{align*}
x_{ij}\mapsto
\begin{cases}
    x_{31}+x_{12}x_{23}-qx_{13}x_{22} & i=3, j=1\\
    x_{ij} & \text{otherwise}.
\end{cases}
\end{align*}
Then $\phi\in\Aut(B_3)$ and $\phi\restrict{B_2}\in\Aut(B_2)$.
\end{proposition}
\begin{proof}
Clearly, we need only prove $\phi$ is an automorphism, as the second part follows trivially. That $\phi$ is a bijection is simple to see, so we need only show it maps the relations of $B_3$ to 0. The only nontrivial relation is then $x_{11}x_{31}-qx_{31}x_{11}$:
\begin{align*}
    \phi(x_{31}x_{11}) &= x_{31}x_{11}+x_{12}x_{23}x_{11}-qx_{13}x_{22}x_{11}\\
    &= q^{-1}x_{11}x_{31}+x_{12}(x_{11}x_{23}-(q-q^{-1})x_{13}x_{21})\\ 
        &\qquad -qx_{13}(x_{11}x_{22}-(q-q^{-1})x_{12}x_{21})\\
    &= q^{-1}(x_{11}x_{31}+x_{11}x_{12}x_{23}-qx_{11}x_{13}x_{22})\\
    &\qquad -(q-q^{-1})x_{12}x_{13}x_{21}+(q-q^{-1})x_{12}x_{13}x_{21} \\
    &=\phi(q^{-1}x_{11}x_{31}),
\end{align*}
as desired.
\end{proof}

\begin{remark}
Let $\phi \in \Aut(B_3)$ be as in Proposition \ref{prop.bautos}. Since $B_3$ is fixed by the transpose automorphism $\tau$, one can prove in a manner similar to that of Theorem \ref{thm.free_prod} that $\langle \phi\rangle*\langle \tau\phi\tau\rangle$ is a proper subgroup of $\Aut(B_3)$. While $\tau\phi\tau$ does not restrict to an automorphism of $\Aut(B_2)$, if $q$ is a root of unity with $\ord(q)=m$, then one can still show $\Aut(B_2)$ (properly) contains a free group on two generators using $\phi$ and
\begin{align*}
    \psi:x_{ij}\mapsto
    \begin{cases}
        x_{11}+x_{21}x_{22}^{m-1}x_{23}^{m-1}x_{31} & i=j=1\\
        x_{12}+Z_{22}x_{23}^{m-1}x_{31} & i=1, j=2\\
        x_{ij} & \text{otherwise}.
    \end{cases}
\end{align*}
\end{remark}

We finish by showing that all automorphisms of $C$ are graded. To do this, we will compute the discriminant $d(C/\cnt(C))$ using \cite[Theorem 6.1]{GKM}, and show it is locally dominating. We first need to set up our notation. 

Let $C''$ denote the subalgebra of $C$ generated by $\{x_{22},x_{23},x_{32},x_{33}\}$. Clearly $C'' \iso \sqmatwo$.
Then $C$ is the iterated Ore extension $C''[x_{13};\sigma_1][x_{31};\sigma_2]$, where
\begin{align*}
    \sigma_1(x_{22})&= x_{22}, ~ \sigma_1(x_{23})=qx_{23}, ~ \sigma_1(x_{32})=x_{32}, ~ \sigma_1(x_{33})=qx_{33}, \\ \sigma_2(x_{22})&=x_{22}, ~ \sigma_2(x_{23})=x_{23}, ~ \sigma_2(x_{32})=qx_{32}, ~ \sigma_2(x_{33})=qx_{33}, ~
    \sigma_2(x_{13})=x_{13}.
\end{align*}

As in \cite{GKM}, we say an automorphism $\sigma$ of an algebra $A$ is \emph{inner} if there exists an $a\in A$ such that $xa=a\sigma(x)$ for all $x\in A$.

\begin{proposition}\label{prop.Cautos}
Assume Hypothesis \ref{hyp.sqman}.
\begin{enumerate}
    \item \label{Cauto1} The center of $C$ is the polynomial ring \[ \cnt(C)=\kk[Z_{13},Z_{22},Z_{23},Z_{31},Z_{32},Z_{33}].\]
    \item \label{Cauto2} The automorphisms $\sigma_1^r\in\Aut(C''),\sigma_2^r\in\Aut(C''[x_{13};\sigma_1])$ are not inner for each $1\leq r\leq m-1$.
    \item \label{Cauto3} The discriminant of $C$ over its center is
\begin{align*}
    d(C/\cnt(C))=_{\cnt(C)^\times}
    \left( x_{13}^{m^6}x_{23}^{m^5}x_{32}^{m^5}x_{31}^{m^6}(x_{22}x_{33}-qx_{23}x_{32})^{m^5}\right)^{m-1}.
\end{align*}
    \item \label{Cauto4} The discriminant $d(C/\cnt(C))$ is locally dominating. 
    \item \label{Cauto5} $\Aut(C)=\Aut_{\gr}(C)$.
\end{enumerate}
\end{proposition}

\begin{proof}
\eqref{Cauto1} That the center is polynomial follows as in Proposition \ref{prop.B1_center}. The details are left to the reader.

\eqref{Cauto2} If some power $\sigma_1^r$, $1\leq r\leq m-1$, was an inner automorphism of $C''$, then we would have an element $0\neq a\in\sqmatwo$ so that $x_{23}a=q^rax_{23}$ and $x_{32}a=ax_{32}$, which is clearly impossible.

Let $C'=C''[x_{13};\sigma_1]$. We claim that no power $\sigma_2^r$ is inner for any $1\leq r\leq m-1$. Suppose there is an $0\neq a\in C'$ such that $Xa=a\sigma_2^r(X)$ for all $X\in C'$. Write $a$ as a sum of basis elements: $a=\sum\gamma x_{13}^{b_1}x_{22}^{b_2}x_{23}^{b_3}x_{32}^{b_4}x_{33}^{b_5}$. From $X=x_{23},x_{32}$ respectively, we obtain
\begin{align*}
    b_1+b_2&\equiv b_5 \text{ mod } m;\\
    r-b_5&\equiv -b_2 \text{ mod } m.
\end{align*}
for all terms of $a$. On the other hand, if $X=x_{22}$, then by comparing the leading terms with respect to the lexicographical ordering of the basis elements such that $x_{13}>x_{22}>x_{23}>x_{32}>x_{33}$, we find that $b_3+b_4\equiv 0 \text{ mod } m$ for the leading term of $a$; a similar argument for $X=x_{33}$ gives $b_1+b_3+b_4\equiv m-r \text{ mod } m$ for the leading term of $a$. Putting these equations together, it follows that $2r\equiv 0 \text{ mod } m$ for the leading term of $a$, which is impossible since $m$ is odd.

\eqref{Cauto3} Let $S'=\cnt(C)\cap (C')^{\sigma_2}$. By part \eqref{Cauto1}, $S'=\kk[Z_{13},Z_{22},Z_{23},Z_{32},Z_{33}]$. Moreover, by part \eqref{Cauto2}, we can apply \cite[Theorem 6.1]{GKM} to obtain
\[ d(C/\cnt(C))=_{\cnt(C)^\times}(d(C'/S'))^m(x_{31}^{m-1})^{m^6}.\]
Denoting $S=S'\cap (C'')^{\sigma_2}=\kk[Z_{22},Z_{23},Z_{32},Z_{33}]$, another application of \cite[Theorem 6.1]{GKM} gives
\[ d(C'/S')=_{(S')^\times}(d(C/S))^m(x_{13}^{m-1})^{m^5}.\]
Next, observe that $d(C/S)$ is one of the discriminants computed in \cite[Theorem 5.7]{NTY}. Hence
\begin{align*}
    d(C/\cnt(C))&=_{\cnt(C)^\times}
    \left(
    \left(\left(x_{32}^{m^3}x_{23}^{m^3}(x_{22}x_{33}-qx_{23}x_{32})^{m^3}\right)^mx_{13}^{m^5}\right)^mx_{31}^{m^6}\right)^{m-1}
    \\
    &=_{\cnt(C)^\times}
    \left( x_{13}^{m^6}x_{23}^{m^5}x_{32}^{m^5}x_{31}^{m^6}(x_{22}x_{33}-qx_{23}x_{32})^{m^5} \right)^{m-1}
\end{align*}

\eqref{Cauto4} This follows from the same argument as \cite[Proposition 5.10]{NTY}. 

\eqref{Cauto5} This follows by \cite[Theorem 2.7]{CPWZ1}.
\end{proof}

\begin{appendix}
\section{The $3 \times 3$ multi-parameter case}\label{sec.mpn3}

In this case, due to the number of parameters involved, we do not attempt to give criteria for the center of $\pmathree$ to be a polynomial ring. Under the assumption that the center is a polynomial ring, we compute the discriminant in a similar fashion to the $2 \times 2$ case and explicitly compute the automorphism group. Many of the computations in this section were done using Macaulay2.

\begin{hypothesis}\label{hyp.pma3}
Set $M=\pmathree$, $Z=\cnt(M)$, and $D=D_{\lambda,\bp}$. Assume that there is some positive integer $\ell$ such that $Z = \kk[y_{ij} \mid 1 \leq i,j \leq 3]$ where $y_{ij}=x_{ij}^\ell$.
\end{hypothesis}

When $\ell=\lcm(\ord(p_{ij}),\ord(\lambda p_{ji}))$ then it is clear that $\kk[y_{ij}]$ is a subring of the center. However we do not have an example that shows that this is the whole center.

The next lemma is actually implied by \cite[Corollary 4.7]{haynal}. However, as it will be useful in the discriminant computation, we work through the details below.

\begin{lemma}\label{lem.subalgs}
Assume Hypothesis \ref{hyp.pma3}.
Let $Y=\{ y_{11}^{u_1}y_{12}^{u_2}y_{21}^{u_3}A(3,3)^{\ell u_4} \mid u_i \in \NN\}$. Then the localization $MY\inv$ is isomorphic to a localization of a quantum affine space.
\end{lemma}


\begin{proof}
We start by setting $M_0$ to be the subalgebra of $M$ generated by $\{x_{11},x_{12},x_{21},x_{22}\}$. Clearly, $M_0$ is isomorphic to some $2 \times 2$ quantum matrix algebra. Now set
\begin{align*}
M_1&=M_0[x_{13};\sigma_1], &
M_2&=M_1[x_{31};\sigma_2], \quad
M_3=M_2[x_{32};\sigma_3,\delta_3], \\
M_4&=M_3[x_{23};\sigma_4,\delta_4], &
M_5&=M_4[x_{33};\sigma_5,\delta_5],
\end{align*}
where the $\sigma_i$ and $\delta_i$ are defined in the obvious way so that $M_5=M$. It is clear that $Y$ is an Ore set in each $M_i$, so we set $M_i'=M_iY\inv$ for $i=0,1,\hdots,5$. Of course, the $\sigma_i$ and $\delta_i$ extend in the natural way to these localizations. By Lemma \ref{lem.local2}, $M_2'$ is isomorphic to a localization of a quantum affine space.


Set $\omega_3 = p_{21} x_{21}\inv x_{22} x_{31}$. We claim that $\omega_3 x_{21} - \sigma_3(x_{21}) \omega_3=\delta_3(x_{ij})$ for all generators $x_{ij}$ in $M_2'$, so that $\delta_3$ is an inner $\sigma_3$-derivation.  Multiplying both sides by $x_{21}$ gives
\[ p_{21}x_{22}x_{31}x_{ij}x_{21} - \lambda p_{32}x_{21} \sigma_3(x_{ij}) x_{22} x_{31} = x_{21}\delta_3(x_{ij})x_{21},\]
and this equation we can verify in Macaulay2 for each $x_{ij}$. Set $x_{32}'=x_{32}-\omega_3$. Then it is clear that $M_3'=M_2'[x_{32}';\sigma_3]$ is a localization of a quantum affine space.

Set $\omega_4=p_{32} x_{12}\inv x_{13}x_{22}$. A completely analogous computation to the above for $M_3'$ applies here to show that $\delta_4$ is an inner $\sigma_4$-derivation. The only novel computation is the following,
\begin{align*}
\omega_4 x_{32} - \sigma_4(x_{32}) \omega_4
	&= (p_{32} x_{12}\inv x_{13}x_{22}) x_{32} - \lambda\inv p_{23}^2 x_{32} (p_{32} x_{12}\inv x_{13}x_{22}) = \delta_4(x_{32}).
\end{align*}

By quantum Laplace expansion we have
\[D = p_{13}p_{23} x_{31} A_{31} - p_{21}p_{13}p_{23} x_{32} A_{32} + x_{33} A_{33}.\]
Set,
\[ z = D - p_{13}p_{23} x_{31} A_{31} + p_{21}p_{13}p_{23} x_{32} A_{32}\]
and then set $\omega_5 = z A_{33}\inv$. We will show that $\delta_5$ is an inner $\sigma_5$-derivation. In this case we rely heavily on Macaulay2 calculations. For example, to check that $\omega_5 x_{11}-\sigma_5(x_{11}) \omega_5 = \delta_5(x_{11})$, we clear denominators and check,
\begin{align*}
(\omega_5 x_{11} - \sigma_5(x_{11}) \omega_5) A_{33}
	= zA_{33}\inv (A_{33} x_{11}) - x_{11} z
	= zx_{11}-x_{11}z
	= \delta_5(x_{11}) A_{33}.
\end{align*}
A similar computation applies to the other generators of $M_4$. Replace $x_{33}$ with $x_{33}'=x_{33}-\omega_5$ to obtain $M_5'=M_4'[x_{33}';\sigma_5]$. Thus, $M_5$ is a localization of a quantum affine space.
\end{proof}

The next result is similar to Lemma \ref{lem.n2disc}.

\begin{lemma}\label{lem.n3disc}
Assume Hypothesis \ref{hyp.pma3}. Then
\[ d(M/Z) =_{\kk^\times} (x_{13}x_{31} A(1,3) A(3,1) D)^{\ell^9(\ell-1)}\]
\end{lemma}
\begin{proof}
As in the proof of Lemma \ref{lem.n2disc} and by our computations in Lemma \ref{lem.subalgs}, 
\[ d(MY\inv/ZY\inv) =_{(ZY\inv)^\times} (x_{13}x_{31} A(1,3) A(3,1) D)^{\ell^9(\ell-1)}.\]
We obtain the same result by replacing $Y$ with
$\{ y_{23}^{u_1}y_{32}^{u_2}y_{33}^{u_3}A(1,1)^{\ell u_4} \mid u_i \in \NN\}$.

Clearly $M$ is a prime $\kk$-algebra and $X=\Spec Z$ is an affine $\kk$-variety. Thus, $(M,Z)$ satisfies \cite[Hypothesis 2.1]{CGWZ1}. Let $U_1$ (resp. $U_2$) be the open subset of $X$ with $a=y_{11}y_{12}y_{21}A(3,3)^{\ell} \neq 0$ (resp. $b=y_{23}y_{32}y_{33}A(1,1)^{\ell} \neq 0$). Let $I$ be the ideal in $Z$ generated by $a$ and $b$. As in Lemma \ref{lem.n2disc} we have $\codim(X \backslash (U_1 \cup U_2) )=2$. Thus, by \cite{CGWZ1}, the discriminant on $U_1 \cup U_2$ extends to a discriminant on $X$.
\end{proof}

Unlike in the $n=2$ case, we do not know if the discriminant is dominating. However, following the analysis of Launois and Lenagan \cite{LL1}, we will establish that the automorphism group in this case is graded. Several of the more arduous calculations were done using Macaulay2.

\begin{theorem}\label{thm.n3auto}
Assume Hypothesis \ref{hyp.pma3}. Then $\Aut(M)=\Aut_{\gr}(M)$.
\end{theorem}
\begin{proof}
Let $\sigma \in \Aut(M)$. Set $d_{ij}=\deg(\sigma(x_{ij})$ and $a_{ij}=\deg(\sigma(A(i,j)))$. We claim $d_{ij}=1$ for all $i,j$. First we show that $d_{11}=d_{33}=1$.

Since $\sigma$ fixes $d(M/Z)$ up to scalar and because $Z$ is a UFD, we may assume that the degrees of all of the factors are preserved. Hence, $\deg(\sigma(D))=3$, $d_{13}=d_{31} = 1$, and $a_{13}=a_{31}=2$. Applying $\sigma$ to $A(3,1)$ gives,
\[ \sigma(A(3,1)) = \sigma(x_{12})\sigma(x_{23})-p_{32}\sigma(x_{13})\sigma(x_{22}).\]
Comparing degrees gives 
\begin{align}\label{eq.deg1}
d_{12}+d_{23}=1+d_{22}.
\end{align} 
A similar computation with $A(1,3)$ gives \begin{align}\label{eq.deg2} d_{21}+d_{32}=1+d_{22}.
\end{align}

Assume $d_{11}+d_{33}>2$. If $d_{11}+d_{22} \leq d_{12}+d_{21}$ and $d_{22}+d_{33}\leq d_{23}+d_{32}$, then
\[ d_{11}+2d_{22}+d_{33} \leq d_{12}+d_{21}+d_{23}+d_{32} = 2+2d_{22},\]
so $d_{11}=d_{33}=1$, contradicting our assumption.

Now we may assume either $d_{11}+d_{22} > d_{12}+d_{21}$ or $d_{22}+d_{33}> d_{23}+d_{32}$. WLOG assume the first. Since,
\[ \sigma(A(3,3)) = \sigma(x_{11})\sigma(x_{22})-p_{21}\sigma(x_{12})\sigma(x_{21}),\]
Then $d(\sigma(A(3,3))) = d_{11}+d_{22}$.

We apply $\sigma$ to the relation
\[ x_{22}D = A(3,3)A(1,1) - \lambda^{-2}p_{12}p_{13}p_{23}A(1,3)A(3,1).\]
Comparing degrees and using the above we have
$d_{22}+3 = d_{11}+d_{22}+a_{11}$. Hence, $d_{11}=1$ and $a_{11}=2$.

Quantum Laplace expansion gives
\[ D = x_{11}A(1,1)-p_{21}x_{12}A(1,2)+p_{31}p_{32}x_{13}A(1,3).\]
Since $d_{11}=d_{13}=1$ and $a_{11}=a_{13}=2$, then $d_{12}+a_{12}=3$, so $d_{12}=1$ and $a_{12}=2$. Similarly, we can show that $d_{21}=1$ and $a_{21}=2$.

We have the relation,
\[ p_{13}p_{21}x_{32}D = A(2,3)A(1,1) - \lambda\inv p_{12}p_{13}p_{32}A(1,3)A(2,1).\]
The degrees on the right-hand side are all known. Hence, $d_{32}=1$. Since, $A(1,3) = x_{21}x_{32} - p_{21}x_{22}x_{31}$, then it follows that $d_{22}=1$. But then $d_{11}+d_{22}=2=d_{12}+d_{21}$, a contradiction. Thus, $d_{11}=d_{33}=1$.

Since $A(2,2)=x_{11}x_{33}-p_{31}x_{13}x_{31}$, then $a_{22}=2$.  Applying $\sigma$ to the relation
\[ \lambda p_{21} p_{31} x_{13}D = A(3,2)A(2,1) - \lambda\inv p_{13}p_{23}p_{21} A(2,2)A(3,1)\]
shows that $a_{32}+a_{21}=4$, so $a_{32}=a_{21}=2$. Similarly, $a_{23}=a_{12}=2$.

Quantum Laplace expansion gives
\[ p_{23} A(2,2)x_{21} = \lambda\inv p_{12}p_{13} A(1,2)x_{11} + \lambda A(3,2)x_{31}.\]
Since the degrees on the right-hand side are all known, this implies that $d_{21}=1$ and $a_{22}=2$. Finally, since $A(3,1) = x_{12}x_{23} - p_{32} x_{22}x_{13}$, then it follows that $d_{22}=1$. Using \eqref{eq.deg1} and \eqref{eq.deg2} completes the proof.
\end{proof}
\end{appendix}

\bibliographystyle{amsalpha}

\providecommand{\bysame}{\leavevmode\hbox to3em{\hrulefill}\thinspace}
\providecommand{\MR}{\relax\ifhmode\unskip\space\fi MR }
\providecommand{\MRhref}[2]{%
  \href{http://www.ams.org/mathscinet-getitem?mr=#1}{#2}
}
\providecommand{\href}[2]{#2}

\end{document}